\documentclass[11pt]{amsart}
\usepackage[utf8]{inputenc}
\usepackage{setspace}
\usepackage[margin=1in]{geometry}
\usepackage{cite, hyperref, url}
\usepackage{tikz, graphicx, marvosym, subfigure, float}
\usetikzlibrary{decorations.pathreplacing}
\usepackage{dsfont, amsfonts, amssymb, amsmath, amsthm}
\usepackage{algorithm, algpseudocode}
\numberwithin{equation}{section}
\newtheorem{remark}{Remark}
\newtheorem{theorem}{Theorem}[section]
\newtheorem{conjecture}[theorem]{Conjecture}
\newtheorem{definition}[theorem]{Definition}
\newtheorem{lemma}[theorem]{Lemma}
\newtheorem{proposition}[theorem]{Proposition}

\title{Induced saturation for complete bipartite posets}

\author{Dingyuan Liu}

\address{Dingyuan Liu \newline Karlsruhe Institute of Technology, Englerstraße 2, D-76131 Karlsruhe, Germany}
\email{liu@mathe.berlin}

\begin{document}
\maketitle

\vspace{-0.7em}
\begin{abstract}
Given $s,t\in\mathbb{N}$, a complete bipartite poset $\mathcal{K}_{s,t}$ is a poset whose Hasse diagram consists of $s$ pairwise incomparable vertices in the upper layer and $t$ pairwise incomparable vertices in the lower layer, such that every vertex in the upper layer is larger than all vertices in the lower layer. A family $\mathcal{F}\subseteq2^{[n]}$ is called induced $\mathcal{K}_{s,t}$-saturated if $(\mathcal{F},\subseteq)$ contains no induced copy of $\mathcal{K}_{s,t}$, whereas adding any set from $2^{[n]}\backslash\mathcal{F}$ to $\mathcal{F}$ creates an induced $\mathcal{K}_{s,t}$. Let $\mathrm{sat}^{*}(n,\mathcal{K}_{s,t})$ denote the smallest size of an induced $\mathcal{K}_{s,t}$-saturated family $\mathcal{F}\subseteq2^{[n]}$. It was conjectured that $\mathrm{sat}^{*}(n,\mathcal{K}_{s,t})$ is superlinear in $n$ for certain values of $s$ and $t$. In this paper, we show that $\mathrm{sat}^{*}(n,\mathcal{K}_{s,t})=O(n)$ for all fixed $s,t\in\mathbb{N}$. Moreover, we prove a linear lower bound on $\mathrm{sat}^{*}(n,\mathcal{P})$ for a large class of posets $\mathcal{P}$, particularly for $\mathcal{K}_{s,2}$ with $s\in\mathbb{N}$.
\end{abstract}

\section{Introduction}
A \textit{poset}\footnote{Throughout this paper we only consider finite posets.} is an ordered pair $(\mathcal{P},\preceq_{\mathcal{P}})$, abbreviated as $\mathcal{P}$ for convenience, where $\mathcal{P}$ is the ground set and $\preceq_{\mathcal{P}}$ is some partial order on $\mathcal{P}$. Two elements $x,y\in\mathcal{P}$ are \textit{comparable} if $x\preceq_{\mathcal{P}}y$ or $y\preceq_{\mathcal{P}}x$, otherwise we call them \textit{incomparable}. We say that a poset $\mathcal{Q}$ \textit{contains an (induced) copy of $\mathcal{P}$}, if there exists an injective function $f:\mathcal{P}\to\mathcal{Q}$, such that $f(x)\preceq_{\mathcal{Q}}f(y)$ if (and only if) $x\preceq_{\mathcal{P}}y$. Moreover, we say that $\mathcal{Q}$ is (\textit{induced}) \textit{$\mathcal{P}$-free}, if $\mathcal{Q}$ contains no (induced) copy of $\mathcal{P}$. Let $n\in\mathbb{N}$ and $\mathcal{P}$ be a fixed poset. A family $\mathcal{F}\subseteq2^{[n]}$ is called (\textit{induced}) \textit{$\mathcal{P}$-saturated}, if
\begin{itemize}
\item $(\mathcal{F},\subseteq)$ is (induced) $\mathcal{P}$-free, and
\item $\forall\,F\in2^{[n]}\backslash\mathcal{F}$, the poset $\left(\mathcal{F}\cup\{F\},\subseteq\right)$ contains an (induced) copy of $\mathcal{P}$.
\end{itemize} 
Determining the minimum size of an (induced) $\mathcal{P}$-saturated family is the so-called (\textit{induced}) \textit{poset saturation problem}. The \textit{saturation number} $\mathrm{sat}(n,\mathcal{P})$ is defined as the smallest size of a $\mathcal{P}$-saturated family $\mathcal{F}\subseteq2^{[n]}$. Similarly, the \textit{induced saturation number} $\mathrm{sat}^{*}(n,\mathcal{P})$ is equal to the smallest size of an induced $\mathcal{P}$-saturated family $\mathcal{F}\subseteq2^{[n]}$.

The study of the poset saturation problem was pioneered by Gerbner, Keszegh, Lemons, Palmer, P\'{a}lv\"{o}lgyi, and Patk\'{o}s~\cite{gerbner2013saturating}, see also~\cite{morrison2014sperner,keszegh2021induced,martin2024saturation} for various results. In particular, by analyzing a greedy colex process, Keszegh, Lemons, Martin, P\'{a}lv\"{o}lgyi, and Patk\'{o}s~\cite{keszegh2021induced} showed that $\mathrm{sat}(n,\mathcal{P})\leq2^{\lvert\mathcal{P}\rvert-2}$ for any fixed poset $\mathcal{P}$.

On the other hand, the systematic study of the induced poset saturation problem was initiated by Ferrara, Kay, Kramer, Martin, Reiniger, Smith, and Sullivan~\cite{ferrara2017saturation}. It has garnered extensive attention and subsequent research~\cite{ivan2020saturation,martin2020improved,keszegh2021induced,ivan2022minimal,freschi2023induced,bastide2023exact,bastide2023polynomial,martin2024induced} in recent years. Unlike $\mathrm{sat}(n,\mathcal{P})$, which is bounded for any fixed $\mathcal{P}$, the function $\mathrm{sat}^{*}(n,\mathcal{P})$ exhibits a dichotomy of behavior. Keszegh et al.~\cite{keszegh2021induced} showed that $\mathrm{sat}^{*}(n,\mathcal{P})$ is either upper bounded by some constant depending only on $\mathcal{P}$, or lower bounded by $\log_{2}{n}$. Furthermore, Keszegh et al.~\cite{keszegh2021induced} conjectured the following stronger dichotomy.
\begin{conjecture}[{\hspace{-0.3mm}\cite[Conjecture 1.4]{keszegh2021induced}}]
\label{cje1}
Let $\mathcal{P}$ be a fixed poset. Then either there exists some constant $C_{\mathcal{P}}$ with $\mathrm{sat}^{*}(n,\mathcal{P})\leq{C_{\mathcal{P}}}$, or $\mathrm{sat}^{*}(n,\mathcal{P})\geq{n+1}$ for all $n\in\mathbb{N}$.
\end{conjecture}
Their conjecture, if true, would be the best possible. As shown in~\cite{ferrara2017saturation}, the induced saturation number for certain posets is exactly $n+1$. For example, let $\mathcal{P}$ consist of two incomparable elements. Then $\mathcal{F}\subseteq2^{[n]}$ is induced $\mathcal{P}$-saturated if and only if $\mathcal{F}$ is a maximal family whose members are pairwise comparable, hence, $\lvert\mathcal{F}\rvert=n+1$. Although Conjecture~\ref{cje1} remains open, the aforementioned dichotomy of $\mathrm{sat}^{*}(n,\mathcal{P})$ has been strengthened recently by Freschi, Piga, Sharifzadeh, and Treglown~\cite{freschi2023induced}, who showed that if $\mathrm{sat}^{*}(n,\mathcal{P})$ is not upper bounded by some constant, then $\mathrm{sat}^{*}(n,\mathcal{P})\geq\min\left\{2\sqrt{n},n/2+1\right\}$. Even more recently, Bastide, Groenland, Ivan, and Johnston~\cite{bastide2023polynomial} proved a polynomial upper bound on $\mathrm{sat}^{*}(n,\mathcal{P})$ for any fixed $\mathcal{P}$, where the degree of the polynomial depends on $\lvert\mathcal{P}\rvert$.

A natural way to visualize the posets is through their Hasse diagrams\footnote{Given a poset $\mathcal{P}$, the \textit{Hasse diagram} of $\mathcal{P}$ is obtained by representing each element of $\mathcal{P}$ as a vertex in the plane, and drawing a line segment that goes upwards from $x$ to $y$ whenever $x\preceq_{\mathcal{P}}y$ and there exists no other element $z$ with $x\preceq_{\mathcal{P}}z\preceq_{\mathcal{P}}y$.}. Since every poset can be uniquely determined by the Hasse diagram up to isomorphism, it is easy to see that $\mathcal{F}\subseteq2^{[n]}$ contains an induced copy of $\mathcal{P}$ if and only if there exists $\mathcal{F}'\subseteq\mathcal{F}$, such that $\left(\mathcal{F}',\subseteq\right)$ can be represented by the same Hasse diagram as $\mathcal{P}$. From now on we shall identify the posets with their Hasse diagrams.

The primary objective of this paper is to study the induced saturation problem for complete bipartite posets. Given $s,t\in\mathbb{N}$, a \textit{complete bipartite poset}, denoted by $\mathcal{K}_{s,t}$, is a poset, whose Hasse diagram consists of $s$ pairwise incomparable vertices in the upper layer and $t$ pairwise incomparable vertices in the lower layer, such that every vertex in the upper layer is larger than all vertices in the lower layer (see Figure~\ref{fig1} for an illustration). Note that $\mathrm{sat}^{*}(n,\mathcal{K}_{s,t})=\mathrm{sat}^{*}(n,\mathcal{K}_{t,s})$ holds due to symmetry, without loss of generality we can assume $s\geq{t}$. Moreover, since $\mathcal{K}_{1,1}$ is a poset of two comparable elements, whose induced saturation number is simply $1$, we exclude this trivial case in later discussions.
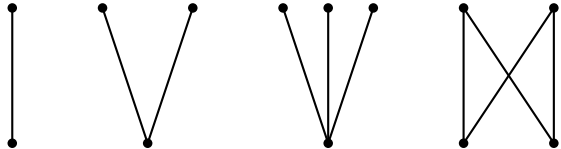
\begin{figure}[H]
\centering
\begin{tikzpicture}[scale=3/5]
\fill (-3,0) circle (3pt);
\fill (-3,-3) circle (3pt);
\draw [thick] (-3,0) -- (-3,-3);
\fill (-1,0) circle (3pt);
\fill (1,0) circle (3pt);
\fill (0,-3) circle (3pt);
\draw [thick] (-1,0) -- (0,-3);
\draw [thick] (1,0) -- (0,-3);
\fill (3,0) circle (3pt);
\fill (4,0) circle (3pt);
\fill (5,0) circle (3pt);
\fill (4,-3) circle (3pt);
\draw [thick] (3,0) -- (4,-3);
\draw [thick] (4,0) -- (4,-3);
\draw [thick] (5,0) -- (4,-3);
\fill (7,0) circle (3pt);
\fill (9,0) circle (3pt);
\fill (7,-3) circle (3pt);
\fill (9,-3) circle (3pt);
\draw [thick] (7,0) -- (7,-3);
\draw [thick] (7,0) -- (9,-3);
\draw [thick] (9,0) -- (7,-3);
\draw [thick] (9,0) -- (9,-3);
\end{tikzpicture}
\caption{The Hasse diagrams of $\mathcal{K}_{1,1}$, $\mathcal{K}_{2,1}$, $\mathcal{K}_{3,1}$, and $\mathcal{K}_{2,2}$.}
\label{fig1}
\end{figure}
As one of the most fundamental types of posets, $\mathcal{K}_{s,t}$ has been studied in the initial work of induced poset saturation. Ferrara et al.~\cite{ferrara2017saturation} showed that $\mathrm{sat}^{*}(n,\mathcal{K}_{2,1})=n+1$ and obtained a linear upper bound
\begin{equation*}
\mathrm{sat}^{*}(n,\mathcal{K}_{s,1})\leq(s-1)n-s+3
\end{equation*} 
for all $n\geq{s}\geq2$. Furthermore, Ferrara et al.~\cite{ferrara2017saturation} showed that $\mathrm{sat}^{*}(n,\mathcal{K}_{2,2})=O(n^{2})$ and conjectured that $\mathrm{sat}^{*}(n,\mathcal{K}_{2,2})=\Theta(n^{2})$. Later Ivan~\cite{ivan2020saturation} proved that $\mathrm{sat}^{*}(n,\mathcal{K}_{s,2})=O(n^{s})$ and $\mathrm{sat}^{*}(n,\mathcal{K}_{s,s})=O(n^{2s-2})$ for all $s\geq2$. In addition, she proposed the following conjecture, which is a strengthening of that by Ferrara et al.~\cite{ferrara2017saturation}.
\begin{conjecture}[{\hspace{-0.3mm}\cite[Conjecture 8 \& Conjecture 9]{ivan2020saturation}}]
\label{cje2}
For all $s\geq2$,
\begin{equation*}
\mathrm{sat}^{*}(n,\mathcal{K}_{s,2})=\Theta(n^{s})\quad\text{and}\quad\mathrm{sat}^{*}(n,\mathcal{K}_{s,s})=\Theta(n^{2s-2}).
\end{equation*}
\end{conjecture}
However, Keszegh et al.~\cite{keszegh2021induced} recently showed that when $n\geq3$, an induced $\mathcal{K}_{2,2}$-saturated family $\mathcal{F}\subseteq2^{[n]}$ of size $6n-10$ can be found through a greedy colex process. This refutes Conjecture~\ref{cje2} in a special case and suggests that Conjecture~\ref{cje2} is probably false for all $s\geq2$ (as remarked in~\cite{ivan2020saturation}). The main result of this paper is the following, which disproves Conjecture~\ref{cje2} in its entirety. We give a linear upper bound on $\mathrm{sat}^{*}(n,\mathcal{K}_{s,t})$ for all fixed $s\geq{t}\geq2$, where we construct an explicit induced $\mathcal{K}_{s,t}$-saturated family using a structure that we call the ``lantern'' (see Figure~\ref{fig2}).
\begin{theorem}
\label{thm1}
Let $n,s,t\in\mathbb{N}$ with $s\geq{t}\geq2$ and $n\geq{2s+t-1}$. Then
\begin{equation*}
\mathrm{sat}^{*}\left(n,\mathcal{K}_{s,t}\right)\leq\left(\binom{s+t-1}{t}(s-1)+\binom{s+t-1}{t-1}(t-1)\right)n+c_{s,t},
\end{equation*}
where $c_{s,t}$ is some constant depending on $s$ and $t$.
\end{theorem}

Ferrara et al.~\cite{ferrara2017saturation} also proved that the induced saturation number is at least $\log_{2}{n}$ for a large class of posets, including the complete bipartite posets (except $\mathcal{K}_{1,1}$). Then, the dichotomy from Freschi et al.~\cite{freschi2023induced} yields a general lower bound $\mathrm{sat}^{*}(n,\mathcal{K}_{s,t})\geq\min\left\{2\sqrt{n},n/2+1\right\}$. If Conjecture~\ref{cje1} was to be true, then one would further expect $\mathrm{sat}^{*}(n,\mathcal{K}_{s,t})\geq{n+1}$. Indeed, the cases of $\mathcal{K}_{1,2}$ and $\mathcal{K}_{2,2}$ have been confirmed by Ferrara et al.~\cite{ferrara2017saturation} and Ivan~\cite{ivan2020saturation}, respectively. Here, we show that this linear lower bound holds for all $\mathcal{K}_{s,2}$.
\begin{theorem}
\label{thm2}
For all $n,s\in\mathbb{N}$, $\mathrm{sat}^{*}\left(n,\mathcal{K}_{s,2}\right)\geq{n+1}$.
\end{theorem}

In fact, we shall prove a stronger version of Theorem~\ref{thm2}. We say that $\mathcal{P}$ is a \textit{poset with legs} if
\begin{itemize}
\item there exist two incomparable elements $a,b\in\mathcal{P}$, such that $a$ and $b$ are smaller than every element in $\mathcal{P}\backslash\{a,b\}$, and
\item there exists an element $c\in\mathcal{P}$ that is larger than $a$ and $b$ but smaller than all elements in $\mathcal{P}\backslash\{a,b,c\}$.
\end{itemize}
The elements $a$ and $b$ are called the \textit{legs} of $\mathcal{P}$, and the element $c$ is referred to as the \textit{hip}. The term ``poset with legs'' was first introduced by Freschi et al.~\cite{freschi2023induced}, where they proved that if $\mathcal{P}$ is a poset with legs, then $\mathcal{P}$ has an induced saturation number at least $n+1$. By generalizing an idea from~\cite{ivan2020saturation}, we show that as long as $\mathcal{P}$ has legs, regardless of whether it has a hip or not, $\mathrm{sat}^{*}(n,\mathcal{P})$ is always lower bounded by $n+1$. This extends the result of Freschi et al.~\cite{freschi2023induced} to a larger class of posets. For example, $\mathcal{K}_{s,2}$ has legs but it is not a poset with legs when $s\geq2$.
\begin{theorem}
\label{thm3}
Let $\mathcal{P}$ be a poset, which contains two incomparable elements $a,b\in\mathcal{P}$, such that $a$ and $b$ are smaller than every element in $\mathcal{P}\backslash\{a,b\}$. Then for $n\in\mathbb{N}$, we have $\mathrm{sat}^{*}(n,\mathcal{P})\geq{n+1}$.
\end{theorem}

Since $\mathcal{K}_{s,2}$ satisfies the above condition, Theorem~\ref{thm3} immediately implies Theorem~\ref{thm2}. Theorems~\ref{thm1} and~\ref{thm2} thereby determine the correct order of $\mathrm{sat}^{*}(n,\mathcal{K}_{s,2})$. Moreover, we conjecture that the upper bound given in Theorem~\ref{thm1} is tight up to a multiplicative constant.
\begin{conjecture}
$\mathrm{sat}^{*}(n,\mathcal{K}_{s,t})=\Theta(n)$ for all fixed $s\geq{t}\geq2$.
\end{conjecture}

For the remainder of this paper, Section~\ref{proofofthm1} is dedicated to proving Theorem~\ref{thm1}, and Theorem~\ref{thm3} is proved in Section~\ref{proofofthm3}.

\subsection*{Notations}
For $m,n\in\mathbb{Z}$, let $[m,n]$ denote the set of all integers $k$ with $m\leq{k}\leq{n}$, particularly, we write $[1,n]$ as $[n]$ for convenience. For $A\subseteq[n]$, denote $A^{c}$ the complement of $A$ with respect to $[n]$. Furthermore, we let $\binom{A}{k}$ denote the family of all $k$-element subsets of $A$, whereupon $2^{A}:=\bigcup_{k=0}^{\left\lvert{A}\right\rvert}\binom{A}{k}$ represents the family of all subsets of $A$.

\section{A linear upper bound for $\mathrm{sat}^{*}(n,\mathcal{K}_{s,t})$}
\label{proofofthm1}
Let $n\in\mathbb{N}$ and $A\subseteq{B}\subseteq[n]$. A poset $\mathcal{C}\subseteq2^{[n]}$ is called a \textit{chain} from $A$ to $B$, if the elements of $\mathcal{C}$ can be written as a sequence from $A$ to $B$, such that every element other than $B$ is a proper subset of the subsequent element. A chain $\mathcal{C}$ from $A$ to $B$ is \textit{complete} if $\lvert\mathcal{C}\rvert=\lvert{B\backslash{A}}\rvert+1$, namely, $\mathcal{C}$ contains sets of cardinality from $\lvert{A}\rvert$ to $\lvert{B}\rvert$. Furthermore, let $\mathcal{C}_{1}$ and $\mathcal{C}_{2}$ be two chains from $A$ to $B$. We say that $\mathcal{C}_{1}$ and $\mathcal{C}_{2}$ are \textit{internally disjoint} if $\mathcal{C}_{1}\cap\mathcal{C}_{2}=\{A,B\}$. The proposition below follows directly from the fact that Boolean lattices are symmetric chain orders. For completeness, we include a short proof.
\begin{proposition}
\label{pro1}
Let $n\in\mathbb{N}$ and $A\subseteq{B}\subseteq[n]$. Then there exist $\lvert{B\backslash{A}}\rvert$ pairwise internally disjoint complete chains from $A$ to $B$.
\end{proposition}
\begin{proof}
Without loss of generality assume $\lvert{B}\backslash{A}\rvert>0$, say $B\backslash{A}=\{x_{1},\dots,x_{k}\}$. For $i\in[k]$, we define the chain $\mathcal{C}_{i}\in2^{[n]}$ as $\left(A,A\cup\{x_{i}\},A\cup\{x_{i},x_{i+1}\},\dots,B\right)$, where the indices are modulo $k$. We can see that $\mathcal{C}_{i}$ is a complete chain from $A$ to $B$ for all $i\in[k]$. Moreover, we fix any $i,j\in[k]$ with $i<j$. Suppose $\mathcal{C}_{i}$ and $\mathcal{C}_{j}$ intersects at some internal element, namely, $\{x_{i},x_{i+1},\dots,x_{i+d}\}=\{x_{j},x_{j+1},\dots,x_{j+d}\}$ for some $d\in[0,k-2]$. Due to $d<k-1$, we have $x_{i-1}\notin\{x_{i},x_{i+1},\dots,x_{i+d}\}$. However, since $i<j$ and $x_{i}\in\{x_{j},x_{j+1},\dots,x_{j+d}\}$, one can deduce that $x_{i-1}\in\{x_{j},x_{j+1},\dots,x_{j+d}\}$, a contradiction.
\end{proof}

\subsection{Construction of an induced $\mathcal{K}_{s,t}$-free family}
\label{construction}
Let $n,s,t\in\mathbb{N}$ with $s\geq{t}\geq2$ and $n\geq{2s+t-1}$. This subsection is devoted to constructing a linear-sized family $\mathcal{F}\subseteq2^{[n]}$ that is induced $\mathcal{K}_{s,t}$-free.

Given $A\subseteq{B}\in2^{[n]}$ with $\lvert{B\backslash{A}}\rvert\geq{k}$, Proposition~\ref{pro1} guarantees that there exists a union of $k$ pairwise internally disjoint complete chains from $A$ to $B$, denoted by $\mathcal{L}$. Let $X_{1},\dots,X_{k}\in\mathcal{L}$ be distinct sets of size $\lvert{A}\rvert+1$. The \textit{first increment set} of $\mathcal{L}$ is defined as $\bigcup_{i=1}^{k}X_{i}\backslash{A}$. Since the chains in $\mathcal{L}$ are pairwise internally disjoint, we know the first increment set of $\mathcal{L}$ has size exactly $k$. Let $Y_{1},\dots,Y_{k}\in\mathcal{L}$ be distinct sets of size $\lvert{B}\rvert-1$. The \textit{last increment set} of $\mathcal{L}$ is defined as $B\backslash\bigcap_{i=1}^{k}Y_{i}$, which also has size $k$.

\begin{definition}
\label{def1}
Let $n,s,t\in\mathbb{N}$ with $s\geq{t}\geq2$ and $n\geq{2s+t-1}$. For $A\subseteq[s+t]$, an \textbf{upper $\boldsymbol{s}$-lantern} $\mathcal{L}^{s}(A)$ is a union of $s-1$ pairwise internally disjoint complete chains from $A$ to $A\cup[s+t+1,n]$, whose last increment set is $[s+t+1,2s+t-1]$. Similarly, a \textbf{lower $\boldsymbol{t}$-lantern} $\mathcal{L}_{t}(A)$ is a union of $t-1$ pairwise internally disjoint complete chains from $A$ to $A\cup[s+t+1,n]$, whose first increment set is $[s+t+1,s+2t-1]$.
\end{definition}
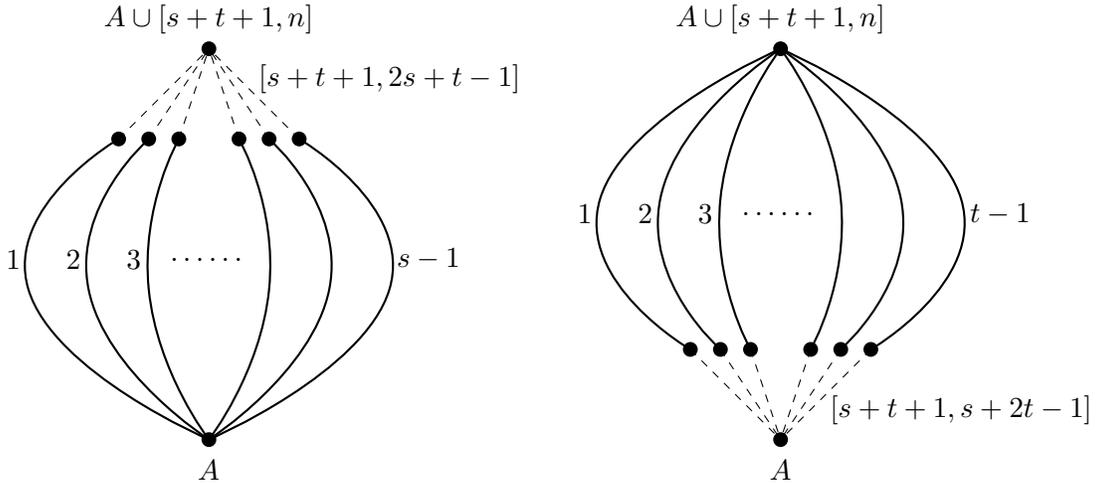
\begin{figure}[H]
\centering
\begin{tikzpicture}[scale=2/5]
\node at (0,-9){$A$};
\node at (0,6){$A\cup[s+t+1,n]$};
\node at (-6.5,-2){$1$};
\node at (-4.5,-2){$2$};
\node at (-2.5,-2){$3$};
\node at (0,-2){$\dots\dots$};
\node at (7.3,-2){$s-1$};
\node at (6,4){$[s+t+1,2s+t-1]$};
\fill (0,5) circle (7pt);
\fill (0,-8) circle (7pt);
\fill (-3,2) circle (7pt);
\fill (-2,2) circle (7pt);
\fill (-1,2) circle (7pt);
\fill (1,2) circle (7pt);
\fill (2,2) circle (7pt);
\fill (3,2) circle (7pt);
\draw[thick] plot [smooth, tension=1] coordinates {(0,-8) (-6,-3) (-3,2)};
\draw[thick] plot [smooth, tension=1] coordinates {(0,-8) (-4,-3) (-2,2)};
\draw[thick] plot [smooth, tension=1] coordinates {(0,-8) (-2,-3) (-1,2)};
\draw[thick] plot [smooth, tension=1] coordinates {(0,-8) (2,-3) (1,2)};
\draw[thick] plot [smooth, tension=1] coordinates {(0,-8) (4,-3) (2,2)};
\draw[thick] plot [smooth, tension=1] coordinates {(0,-8) (6,-3) (3,2)};
\draw[dashed] (-3,2)--(0,5);
\draw[dashed] (-2,2)--(0,5);
\draw[dashed] (-1,2)--(0,5);
\draw[dashed] (1,2)--(0,5);
\draw[dashed] (2,2)--(0,5);
\draw[dashed] (3,2)--(0,5);

\node at (19,-9){$A$};
\node at (19,6){$A\cup[s+t+1,n]$};
\node at (12.5,-0.5){$1$};
\node at (14.5,-0.5){$2$};
\node at (16.5,-0.5){$3$};
\node at (19,-0.5){$\dots\dots$};
\node at (26.3,-0.5){$t-1$};
\node at (25,-7){$[s+t+1,s+2t-1]$};
\fill (19,5) circle (7pt);
\fill (19,-8) circle (7pt);
\fill (16,-5) circle (7pt);
\fill (17,-5) circle (7pt);
\fill (18,-5) circle (7pt);
\fill (20,-5) circle (7pt);
\fill (21,-5) circle (7pt);
\fill (22,-5) circle (7pt);
\draw[thick] plot [smooth, tension=1] coordinates {(16,-5) (13,0) (19,5)};
\draw[thick] plot [smooth, tension=1] coordinates {(17,-5) (15,0) (19,5)};
\draw[thick] plot [smooth, tension=1] coordinates {(18,-5) (17,0) (19,5)};
\draw[thick] plot [smooth, tension=1] coordinates {(20,-5) (21,0) (19,5)};
\draw[thick] plot [smooth, tension=1] coordinates {(21,-5) (23,0) (19,5)};
\draw[thick] plot [smooth, tension=1] coordinates {(22,-5) (25,0) (19,5)};
\draw[dashed] (16,-5)--(19,-8);
\draw[dashed] (17,-5)--(19,-8);
\draw[dashed] (18,-5)--(19,-8);
\draw[dashed] (20,-5)--(19,-8);
\draw[dashed] (21,-5)--(19,-8);
\draw[dashed] (22,-5)--(19,-8);
\end{tikzpicture}
\caption{$\mathcal{L}^{s}(A)$ and $\mathcal{L}_{t}(A)$.}
\label{fig2}
\end{figure}

\begin{remark}
\label{rem1}
Observe that $F\cap[s+t-1]=A\cap[s+t-1]$ holds for any set $F$ in $\mathcal{L}^{s}(A)$ or $\mathcal{L}_{t}(A)$.
\end{remark}

The restriction on the first or last increment set in Definition~\ref{def1} seems peculiar at this point. We only need it to help us deal with some critical cases in the proof later (see Lemmas~\ref{lem5} and~\ref{lem6}). Other than that, we can safely ignore the restriction and consider the lantern as a simple union of pairwise internally disjoint chains.

Our construction of an induced $\mathcal{K}_{s,t}$-free family consists of the following parts:
\begin{itemize}
\item $\mathcal{F}_{1}:=\left\{[n]\right\}\cup\left\{\{x\}^{c}\text{ for all }x\in[s+t-1]\right\}$,
\item $\mathcal{F}_{2}:=\bigcup_{A\in\binom{[s+t-1]}{t}}\mathcal{L}^{s}(A)$,
\item $\mathcal{F}_{3}:=\bigcup_{A\in\binom{[s+t-1]}{t-1}}\mathcal{L}_{t}\left(A\cup\{s+t\}\right)$,
\item $\mathcal{F}_{4}:=\left\{\emptyset\right\}\cup\left\{\{x\}\text{ for all }x\in[s+t-1]\right\}$.
\end{itemize}
In particular, $\sum_{i=1}^{4}\lvert\mathcal{F}_{i}\rvert\leq\left(\binom{s+t-1}{t}(s-1)+\binom{s+t-1}{t-1}(t-1)\right)n$ follows by a simple calculation. Let $\mathcal{F}':=\bigcup_{i=1}^{4}\mathcal{F}_{i}$. We shall first show that $\mathcal{F}'$ does not contain an induced copy of $\mathcal{K}_{s,t}$.

\begin{lemma}
\label{lem1}
Let $F_{1},\dots,F_{s}\in\mathcal{F}'$ be pairwise incomparable. Then
\begin{equation*}
\left\lvert\bigcap_{i=1}^{s}\left(F_{i}\cap[s+t-1]\right)\right\rvert<t.
\end{equation*}
\end{lemma}
\begin{proof}
Suppose there exist $F_{1},\dots,F_{s}\in\mathcal{F}'$ that are pairwise incomparable and $\bigcap_{i=1}^{s}\left(F_{i}\cap[s+t-1]\right)$ has size at least $t$. Then we must have $\left\{F_{1},\dots,F_{s}\right\}\cap\mathcal{F}_{4}=\emptyset$, since for any set in $\mathcal{F}_{4}$, its intersection with $[s+t-1]$ has size smaller than $t$. Also, $\left\{F_{1},\dots,F_{s}\right\}\cap\mathcal{F}_{3}=\emptyset$, because every set in $\mathcal{F}_{3}$ belongs to a lantern $\mathcal{L}_{t}\left(A\cup\{s+t\}\right)$ for some $A\in\binom{[s+t-1]}{t-1}$ and Remark~\ref{rem1} indicates that its intersection with $[s+t-1]$ is exactly $A$ and thus has size smaller than $t$.\\
We claim that $\left\{F_{1},\dots,F_{s}\right\}\cap\mathcal{F}_{1}\neq\emptyset$ and $\left\{F_{1},\dots,F_{s}\right\}\cap\mathcal{F}_{2}\neq\emptyset$ can not be simultaneously true. Indeed, assume without loss of generality that $F_{1}\in\mathcal{F}_{1}$ and $F_{2}\in\mathcal{F}_{2}$. Since $F_{2}$ belongs to a lantern $\mathcal{L}^{s}(A)$ for some $A\in\binom{[s+t-1]}{t}$, we have $F_{2}\cap[s+t-1]=A$. On the other hand, because $F_{1}$ is incomparable with $F_{2}$, the only possibility is $F_{1}=\{x\}^{c}$ with some $x\in{A}$. But then it holds that $F_{1}\cap{F_{2}}\cap[s+t-1]=A\backslash\{x\}$, which is a contradiction since $\lvert{A\backslash\{x\}}\rvert<t$.\\
Hence, we are left with the following two cases.\\
\textbf{Case 1:} $\left\{F_{1},\dots,F_{s}\right\}\subseteq\mathcal{F}_{1}$.\\
Since $F_{1},\dots,F_{s}$ are pairwise incomparable, we must have $F_{i}=\{x_{i}\}^{c}$ with $x_{i}\in[s+t-1]$ for all $i\in[s]$. Then it follows that $\big\lvert\bigcap_{i=1}^{s}\left(F_{i}\cap[s+t-1]\right)\big\rvert=\big\lvert[s+t-1]\backslash\{x_{1},\dots,x_{s}\}\big\rvert<t$, a contradiction.\\
\textbf{Case 2:} $\left\{F_{1},\dots,F_{s}\right\}\subseteq\mathcal{F}_{2}$.\\
Observe that $F_{1},\dots,F_{s}$ must belong to the same lantern $\mathcal{L}^{s}(A)$. Otherwise, assume without loss of generality that $F_{1}\in\mathcal{L}^{s}(A_{1})$ and $F_{2}\in\mathcal{L}^{s}(A_{2})$, where $A_{1},A_{2}\in\binom{[s+t-1]}{t}$ and $A_{1}\neq{A_{2}}$. Then we would have $\big\lvert{F_{1}\cap{F_{2}}\cap[s+t-1]}\big\rvert=\lvert{A_{1}\cap{A_{2}}}\rvert<t$, a contradiction. Now, because $F_{1},\dots,F_{s}$ belong to the same lantern $\mathcal{L}^{s}(A)$, by the pigeonhole principle, at least two of them lie in the same chain, meaning that $F_{1},\dots,F_{s}$ are not pairwise incomparable, a contradiction.
\end{proof}

\begin{lemma}
\label{lem2}
Let $F_{1},\dots,F_{t}\in\mathcal{F}'$ be pairwise incomparable. Then
\begin{equation*}
\left\lvert\bigcup_{i=1}^{t}\left(F_{i}\cap[s+t-1]\right)\right\rvert\geq{t}.
\end{equation*}
\end{lemma}
\begin{proof}
Suppose there exist $F_{1},\dots,F_{t}\in\mathcal{F}'$ that are pairwise incomparable and $\bigcup_{i=1}^{t}\left(F_{i}\cap[s+t-1]\right)$ has size smaller than $t$. Then we immediately have $\left\{F_{1},\dots,F_{t}\right\}\cap\left(\mathcal{F}_{1}\cup\mathcal{F}_{2}\right)=\emptyset$, as every set in $\left(\mathcal{F}_{1}\cup\mathcal{F}_{2}\right)$ contains at least $t$ elements from $[s+t-1]$.\\
We claim next that $\left\{F_{1},\dots,F_{t}\right\}\cap\mathcal{F}_{3}\neq\emptyset$ and $\left\{F_{1},\dots,F_{t}\right\}\cap\mathcal{F}_{4}\neq\emptyset$ can not hold simultaneously. Otherwise, assume without loss of generality that $F_{3}\in\mathcal{F}_{3}$ and $F_{4}\in\mathcal{F}_{4}$. Since $F_{3}$ belongs to a lantern $\mathcal{L}_{t}\left(A\cup\{s+t\}\right)$ for some $A\in\binom{[s+t-1]}{t-1}$, $F_{3}\cap[s+t-1]=A$. Then, to assure that $F_{4}\in\mathcal{F}_{4}$ is incomparable with $F_{3}$, one must have $F_{4}=\{x\}$ with some $x\in[s+t-1]\backslash{A}$. This yields that $\left(F_{3}\cup{F_{4}}\right)\cap[s+t-1]=A\cup\{x\}$, a contradiction.\\
We are again left with two cases.\\
\textbf{Case 1:} $\left\{F_{1},\dots,F_{t}\right\}\subseteq\mathcal{F}_{3}$.\\
If $F_{1}\in\mathcal{L}_{t}\left(A_{1}\cup\{s+t\}\right)$ and $F_{2}\in\mathcal{L}_{t}\left(A_{2}\cup\{s+t\}\right)$ with $A_{1}\neq{A_{2}}\in\binom{[s+t-1]}{t-1}$, then we would have $\big\lvert\left(F_{1}\cup{F_{2}}\right)\cap[s+t-1]\big\rvert=\lvert{A_{1}\cup{A_{2}}}\rvert\geq{t}$, a contradiction. Accordingly, $F_{1},\dots,F_{s}$ must belong to the same lantern $\mathcal{L}_{t}\left(A\cup\{s+t\}\right)$, by the pigeonhole principle, at least two of them lie in the same chain. Hence, $F_{1},\dots,F_{t}$ are not pairwise incomparable, a contradiction.\\
\textbf{Case 2:} $\left\{F_{1},\dots,F_{t}\right\}\subseteq\mathcal{F}_{4}$.\\
Recall that $F_{1},\dots,F_{t}$ are pairwise incomparable, the only possibility is $F_{i}=\{x_{i}\}$ with $x_{i}\in[s+t-1]$ for all $i\in[t]$. Then we have $\big\lvert\bigcup_{i=1}^{t}\left(F_{i}\cap[s+t-1]\right)\big\rvert=\big\lvert\{x_{1},\dots,x_{t}\}\big\rvert=t$, leading to a contradiction.
\end{proof}

\begin{lemma}
\label{lem3}
The family $\mathcal{F}'$ is induced $\mathcal{K}_{s,t}$-free.
\end{lemma}
\begin{proof}
Suppose there exists an induced copy of $\mathcal{K}_{s,t}$ in $\mathcal{F}'$, whose upper layer consists of $F_{1},\dots,F_{s}$ and whose lower layer consists of $F_{1}',\dots,F_{t}'$, see Figure~\ref{fig3}.
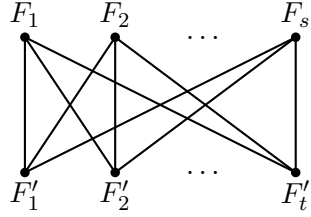
\begin{figure}[H]
\centering
\begin{tikzpicture}[scale=3/5]
\node at (-3,0.5){$F_{1}$};
\fill (-3,0) circle (3pt);
\node at (-1,0.5){$F_{2}$};
\fill (-1,0) circle (3pt);
\node at (1,0){$\dots$};
\node at (3,0.5){$F_{s}$};
\fill (3,0) circle (3pt);
\node at (-3,-3.5){$F_{1}'$};
\fill (-3,-3) circle (3pt);
\node at (-1,-3.5){$F_{2}'$};
\fill (-1,-3) circle (3pt);
\node at (1,-3){$\dots$};
\node at (3,-3.5){$F_{t}'$};
\fill (3,-3) circle (3pt);
\draw [thick] (-3,0) -- (-3,-3);
\draw [thick] (-3,0) -- (-1,-3);
\draw [thick] (-3,0) -- (3,-3);
\draw [thick] (-1,0) -- (-3,-3);
\draw [thick] (-1,0) -- (-1,-3);
\draw [thick] (-1,0) -- (3,-3);
\draw [thick] (3,0) -- (-3,-3);
\draw [thick] (3,0) -- (-1,-3);
\draw [thick] (3,0) -- (3,-3);
\end{tikzpicture}
\caption{The hypothetical induced copy of $\mathcal{K}_{s,t}$.}
\label{fig3}
\end{figure}
\noindent
Then it holds that $\bigcup_{i=1}^{t}F_{i}'\subseteq\bigcap_{i=1}^{s}F_{i}$. But by Lemmas~\ref{lem1} and~\ref{lem2} we have $\left\lvert\bigcap_{i=1}^{s}\left(F_{i}\cap[s+t-1]\right)\right\rvert<t$ and $\left\lvert\bigcup_{i=1}^{t}\left(F_{i}'\cap[s+t-1]\right)\right\rvert\geq{t}$, implying that $\bigcup_{i=1}^{t}F_{i}'\not\subseteq\bigcap_{i=1}^{s}F_{i}$, a contradiction.
\end{proof}

Now that $\bigcup_{i=1}^{4}\mathcal{F}_{i}$ is induced $\mathcal{K}_{s,t}$-free, we construct the last part $\mathcal{F}_{5}$ by letting $\mathcal{F}_{5}$ be a maximal subset of $\mathcal{G}_{1}\cup\mathcal{G}_{2}$, where
\begin{equation*}
\mathcal{G}_{1}:=\bigcup_{i=2}^{s}\left\{A^{c}\text{ for all }A\in\binom{[2s+t-1]}{i}\right\}\quad\text{and}\quad\mathcal{G}_{2}:=\bigcup_{i=2}^{t}\binom{[s+2t-1]}{i},
\end{equation*}
such that there is no induced copy of $\mathcal{K}_{s,t}$ in $\bigcup_{i=1}^{5}\mathcal{F}_{i}$. Let $\mathcal{F}:=\bigcup_{i=1}^{5}\mathcal{F}_{i}$ be our final family. It is obvious that $\mathcal{F}$ is induced $\mathcal{K}_{s,t}$-free.

\subsection{Induced $\mathcal{K}_{s,t}$-saturated}
The purpose of this subsection is to show that the constructed family $\mathcal{F}\subseteq2^{[n]}$ is indeed induced $\mathcal{K}_{s,t}$-saturated. Namely, for every $F\in2^{[n]}\backslash\mathcal{F}$, we shall prove that there exists an induced copy of $\mathcal{K}_{s,t}$ in $\mathcal{F}\cup\{F\}$.
\begin{lemma}
\label{lem4}
Let $F\in2^{[n]}\backslash\mathcal{F}$ with $\big\lvert{F\cap[s+t-1]}\big\rvert\in\{0,s+t-1\}$. Then $\mathcal{F}\cup\{F\}$ contains an induced copy of $\mathcal{K}_{s,t}$.
\end{lemma}
\begin{proof}
We split the proof into two cases.\\
\textbf{Case 1:} $\big\lvert{F\cap[s+t-1]}\big\rvert=0$.\\
Since $F$ is nonempty and $F\cap[s+t-1]=\emptyset$, $F$ is incomparable with any singleton $\{x\}\subseteq[s+t-1]$. Hence we can take $\{1\},\dots,\{t-1\}\in\mathcal{F}_{4}$, so that $F,\{1\},\dots,\{t-1\}$ are pairwise incomparable, which shall form the lower layer of our desired $\mathcal{K}_{s,t}$. For the upper layer, we choose $\{t\}^{c},\dots,\{s+t-1\}^{c}\in\mathcal{F}_{1}$. In this way we obtain an induced copy of $\mathcal{K}_{s,t}$ in $\mathcal{F}\cup\{F\}$, see Figure~\ref{fig4}.
\begin{figure}[H]
\begin{minipage}{0.49\textwidth}
\centering
\begin{tikzpicture}[scale=3/5]
\node at (-3,0.5){$\{t\}^{c}$};
\fill (-3,0) circle (3pt);
\node at (-1,0.5){$\{t+1\}^{c}$};
\fill (-1,0) circle (3pt);
\node at (1,0){$\dots$};
\node at (3,0.5){$\{s+t-1\}^{c}$};
\fill (3,0) circle (3pt);
\node at (-3,-3.5){$F$};
\fill (-3,-3) circle (3pt);
\node at (-1,-3.5){$\{1\}$};
\fill (-1,-3) circle (3pt);
\node at (1,-3){$\dots$};
\node at (3,-3.5){$\{t-1\}$};
\fill (3,-3) circle (3pt);
\draw [thick] (-3,0) -- (-3,-3);
\draw [thick] (-3,0) -- (-1,-3);
\draw [thick] (-3,0) -- (3,-3);
\draw [thick] (-1,0) -- (-3,-3);
\draw [thick] (-1,0) -- (-1,-3);
\draw [thick] (-1,0) -- (3,-3);
\draw [thick] (3,0) -- (-3,-3);
\draw [thick] (3,0) -- (-1,-3);
\draw [thick] (3,0) -- (3,-3);
\end{tikzpicture}
\caption{The induced copy of $\mathcal{K}_{s,t}$ in Case 1.}
\label{fig4}
\end{minipage}
\hfill
\begin{minipage}{0.49\textwidth}
\centering
\begin{tikzpicture}[scale=3/5]
\node at (9,0.5){$F$};
\fill (9,0) circle (3pt);
\node at (11,0.5){$\{1\}^{c}$};
\fill (11,0) circle (3pt);
\node at (13,0){$\dots$};
\node at (15,0.5){$\{s-1\}^{c}$};
\fill (15,0) circle (3pt);
\node at (9,-3.5){$\{s\}$};
\fill (9,-3) circle (3pt);
\node at (11,-3.5){$\{s+1\}$};
\fill (11,-3) circle (3pt);
\node at (13,-3){$\dots$};
\node at (15,-3.5){$\{s+t-1\}$};
\fill (15,-3) circle (3pt);
\draw [thick] (9,0) -- (9,-3);
\draw [thick] (9,0) -- (11,-3);
\draw [thick] (9,0) -- (15,-3);
\draw [thick] (11,0) -- (9,-3);
\draw [thick] (11,0) -- (11,-3);
\draw [thick] (11,0) -- (15,-3);
\draw [thick] (15,0) -- (9,-3);
\draw [thick] (15,0) -- (11,-3);
\draw [thick] (15,0) -- (15,-3);
\end{tikzpicture}
\caption{The induced copy of $\mathcal{K}_{s,t}$ in Case 2.}
\label{fig5}
\end{minipage}
\end{figure}
\noindent
\textbf{Case 2:} $\big\lvert{F\cap[s+t-1]}\big\rvert=s+t-1$.\\
We first take $\{1\}^{c},\dots,\{s-1\}^{c}\in\mathcal{F}_{1}$. Since $F\neq[n]$, the sets $F,\{1\}^{c},\dots,\{s-1\}^{c}$ are pairwise incomparable, forming the upper layer of the desired $\mathcal{K}_{s,t}$. On the other hand, we take $\{s\},\dots,\{s+t-1\}\in\mathcal{F}_{4}$ to form the lower layer. This gives us an induced copy of $\mathcal{K}_{s,t}$ in $\mathcal{F}\cup\{F\}$, as shown in Figure~\ref{fig5}.
\end{proof}

\begin{lemma}
\label{lem5}
Let $F\in2^{[n]}\backslash\mathcal{F}$ with $1\leq\big\lvert{F\cap[s+t-1]}\big\rvert\leq{t-1}$. Then $\mathcal{F}\cup\{F\}$ contains an induced copy of $\mathcal{K}_{s,t}$.
\end{lemma}
\begin{proof}
First we can assume without loss of generality that $F\notin\mathcal{G}_{1}\cup\mathcal{G}_{2}$, because by our construction of $\mathcal{F}_{5}$, $F\in\left(\mathcal{G}_{1}\cup\mathcal{G}_{2}\right)\backslash\mathcal{F}$ would imply that $\mathcal{F}\cup\{F\}$ contains an induced copy of $\mathcal{K}_{s,t}$. Let $A\in\binom{[s+t-1]}{t-1}$ with $F\cap[s+t-1]\subseteq{A}$.\\
\textbf{Case 1:} $\lvert{F}\rvert\leq{t}$.\\
Since $F\notin\mathcal{F}_{4}$, we actually have $2\leq\lvert{F}\rvert\leq{t}$. If $F\backslash[s+2t-1]=\emptyset$, then we have $F\in\binom{[s+2t-1]}{i}$ for some $2\leq{i}\leq{t}$. This contradicts the fact that $F\notin\mathcal{G}_{2}$. Therefore, $F\backslash[s+2t-1]\neq\emptyset$. We take distinct sets $X_{1},\dots,X_{t-1}$ of size $t+1$ from the lantern $\mathcal{L}_{t}\left(A\cup\{s+t\}\right)\subseteq\mathcal{F}_{3}$. Due to the same size, $X_{1},\dots,X_{t-1}$ are pairwise incomparable. We claim that $F,X_{1},\dots,X_{t-1}$ are also pairwise incomparable. Indeed, because of $\lvert{F}\rvert\leq{t}$, $F$ is not a superset of any of $X_{1},\dots,X_{t-1}$. Furthermore, by definition the first increment set of $\mathcal{L}_{t}\left(A\cup\{s+t\}\right)$ is $[s+t+1,s+2t-1]$, which means that $\bigcup_{i=1}^{t-1}X_{i}=A\cup[s+t,s+2t-1]\subseteq[s+2t-1]$ and hence $F$ is not a subset of any of $X_{1},\dots,X_{t-1}$. Therefore, $F,X_{1},\dots,X_{t-1}$ form a lower layer of our desired $\mathcal{K}_{s,t}$. On the other hand, we note that $F\cap[s+t-1]\subseteq{X_{1}\cap[s+t-1]}=\dots=X_{t-1}\cap[s+t-1]=A$ holds by Remark~\ref{rem1}. Thus one can take $\{x_{1}\}^{c},\dots,\{x_{s}\}^{c}\in\mathcal{F}_{1}$ with $x_{i}\in[s+t-1]\backslash{A}$ for all $i\in[s]$ as the upper layer, see Figure~\ref{fig6}.
\begin{figure}[H]
\begin{minipage}{0.49\textwidth}
\centering
\begin{tikzpicture}[scale=3/5]
\node at (-3,0.5){$\{x_{1}\}^{c}$};
\fill (-3,0) circle (3pt);
\node at (-1,0.5){$\{x_{2}\}^{c}$};
\fill (-1,0) circle (3pt);
\node at (1,0){$\dots$};
\node at (3,0.5){$\{x_{s}\}^{c}$};
\fill (3,0) circle (3pt);
\node at (-3,-3.5){$F$};
\fill (-3,-3) circle (3pt);
\node at (-1,-3.5){$X_{1}$};
\fill (-1,-3) circle (3pt);
\node at (1,-3){$\dots$};
\node at (3,-3.5){$X_{t-1}$};
\fill (3,-3) circle (3pt);
\draw [thick] (-3,0) -- (-3,-3);
\draw [thick] (-3,0) -- (-1,-3);
\draw [thick] (-3,0) -- (3,-3);
\draw [thick] (-1,0) -- (-3,-3);
\draw [thick] (-1,0) -- (-1,-3);
\draw [thick] (-1,0) -- (3,-3);
\draw [thick] (3,0) -- (-3,-3);
\draw [thick] (3,0) -- (-1,-3);
\draw [thick] (3,0) -- (3,-3);
\end{tikzpicture}
\caption{The induced copy of $\mathcal{K}_{s,t}$ in Case 1.}
\label{fig6}
\end{minipage}
\hfill
\begin{minipage}{0.49\textwidth}
\centering
\begin{tikzpicture}[scale=3/5]
\node at (9,0.5){$\{x_{1}\}^{c}$};
\fill (9,0) circle (3pt);
\node at (11,0.5){$\{x_{2}\}^{c}$};
\fill (11,0) circle (3pt);
\node at (13,0){$\dots$};
\node at (15,0.5){$\{x_{s}\}^{c}$};
\fill (15,0) circle (3pt);
\node at (9,-3.5){$F$};
\fill (9,-3) circle (3pt);
\node at (11,-3.5){$Y_{1}$};
\fill (11,-3) circle (3pt);
\node at (13,-3){$\dots$};
\node at (15,-3.5){$Y_{t-1}$};
\fill (15,-3) circle (3pt);
\draw [thick] (9,0) -- (9,-3);
\draw [thick] (9,0) -- (11,-3);
\draw [thick] (9,0) -- (15,-3);
\draw [thick] (11,0) -- (9,-3);
\draw [thick] (11,0) -- (11,-3);
\draw [thick] (11,0) -- (15,-3);
\draw [thick] (15,0) -- (9,-3);
\draw [thick] (15,0) -- (11,-3);
\draw [thick] (15,0) -- (15,-3);
\end{tikzpicture}
\caption{The induced $\mathcal{K}_{s,t}$ in Case 2.}
\label{fig7}
\end{minipage}
\end{figure}
\noindent
\textbf{Case 2:} $t<\lvert{F}\rvert<n-s$.\\
We can take distinct sets $Y_{1},\dots,Y_{t-1}$ of size $\lvert{F}\rvert$ from the lantern $\mathcal{L}_{t}\left(A\cup\{s+t\}\right)\subseteq\mathcal{F}_{3}$. Due to the same size, $F,Y_{1},\dots,Y_{t-1}$ are pairwise incomparable and form an adequate lower layer. Since $F\cap[s+t-1]\subseteq{Y_{1}\cap[s+t-1]}=\dots=Y_{t-1}\cap[s+t-1]=A$ still holds, we choose the same upper layer as in the previous case, namely, $\{x_{1}\}^{c},\dots,\{x_{s}\}^{c}\in\mathcal{F}_{1}$ with $x_{i}\in[s+t-1]\backslash{A}$ for all $i\in[s]$, see Figure~\ref{fig7}.\\
\textbf{Case 3:} $\lvert{F}\rvert\geq{n-s}$.\\
In this case, we must have $F=A\cup[s+t,n]=\left(A\cup\{s+t\}\right)\cup[s+t+1,n]$, which implies that $F\in\mathcal{L}_{t}\left(A\cup\{s+t\}\right)\subseteq\mathcal{F}_{3}$, a contradiction.
\end{proof}

\begin{lemma}
\label{lem6}
Let $F\in2^{[n]}\backslash\mathcal{F}$ with $t\leq\big\lvert{F\cap[s+t-1]}\big\rvert<s+t-1$. Then $\mathcal{F}\cup\{F\}$ contains an induced copy of $\mathcal{K}_{s,t}$.
\end{lemma}
\begin{proof}
As previously, we can assume that $F\notin\mathcal{G}_{1}\cup\mathcal{G}_{2}$. Let $A\in\binom{[s+t-1]}{t}$ with $A\subseteq{F\cap[s+t-1]}$.\\
\textbf{Case 1:} $t\leq\lvert{F}\rvert<n-s$.\\
If $\lvert{F}\rvert=t$, then $F=A\in\mathcal{L}^{s}(A)\subseteq\mathcal{F}_{2}$, a contradiction. If $t<\lvert{F}\rvert<n-s$, then we can take distinct sets $X_{1},\dots,X_{s-1}$ of size $\lvert{F}\rvert$ from the lantern $\mathcal{L}_{s}(A)\subseteq\mathcal{F}_{2}$. Due to the same size, $F,X_{1},\dots,X_{s-1}$ are pairwise incomparable and form the upper layer of our desired $\mathcal{K}_{s,t}$. As for the lower layer, we choose the singletons $\{x_{1}\},\dots,\{x_{t}\}\in\mathcal{F}_{4}$ with $x_{i}\in{A}$ for all $i\in[t]$, see Figure~\ref{fig8}.
\begin{figure}[H]
\begin{minipage}{0.49\textwidth}
\centering
\begin{tikzpicture}[scale=3/5]
\node at (-3,0.5){$F$};
\fill (-3,0) circle (3pt);
\node at (-1,0.5){$X_{1}$};
\fill (-1,0) circle (3pt);
\node at (1,0){$\dots$};
\node at (3,0.5){$X_{s-1}$};
\fill (3,0) circle (3pt);
\node at (-3,-3.5){$\{x_{1}\}$};
\fill (-3,-3) circle (3pt);
\node at (-1,-3.5){$\{x_{2}\}$};
\fill (-1,-3) circle (3pt);
\node at (1,-3){$\dots$};
\node at (3,-3.5){$\{x_{t}\}$};
\fill (3,-3) circle (3pt);
\draw [thick] (-3,0) -- (-3,-3);
\draw [thick] (-3,0) -- (-1,-3);
\draw [thick] (-3,0) -- (3,-3);
\draw [thick] (-1,0) -- (-3,-3);
\draw [thick] (-1,0) -- (-1,-3);
\draw [thick] (-1,0) -- (3,-3);
\draw [thick] (3,0) -- (-3,-3);
\draw [thick] (3,0) -- (-1,-3);
\draw [thick] (3,0) -- (3,-3);
\end{tikzpicture}
\caption{The induced copy of $\mathcal{K}_{s,t}$ in Case 1.}
\label{fig8}
\end{minipage}
\hfill
\begin{minipage}{0.49\textwidth}
\centering
\begin{tikzpicture}[scale=3/5]
\node at (9,0.5){$F$};
\fill (9,0) circle (3pt);
\node at (11,0.5){$Y_{1}$};
\fill (11,0) circle (3pt);
\node at (13,0){$\dots$};
\node at (15,0.5){$Y_{s-1}$};
\fill (15,0) circle (3pt);
\node at (9,-3.5){$\{x_{1}\}$};
\fill (9,-3) circle (3pt);
\node at (11,-3.5){$\{x_{2}\}$};
\fill (11,-3) circle (3pt);
\node at (13,-3){$\dots$};
\node at (15,-3.5){$\{x_{t}\}$};
\fill (15,-3) circle (3pt);
\draw [thick] (9,0) -- (9,-3);
\draw [thick] (9,0) -- (11,-3);
\draw [thick] (9,0) -- (15,-3);
\draw [thick] (11,0) -- (9,-3);
\draw [thick] (11,0) -- (11,-3);
\draw [thick] (11,0) -- (15,-3);
\draw [thick] (15,0) -- (9,-3);
\draw [thick] (15,0) -- (11,-3);
\draw [thick] (15,0) -- (15,-3);
\end{tikzpicture}
\caption{The induced copy of $\mathcal{K}_{s,t}$ in Case 2.}
\label{fig9}
\end{minipage}
\end{figure}
\noindent
\textbf{Case 2:} $\lvert{F}\rvert\geq{n-s}$.\\
First we observe that $F^{c}\not\subseteq[2s+t-1]$. Indeed, if $F^{c}$ is a singleton set, given that $\big\lvert{F\cap[s+t-1]}\big\rvert<s+t-1$, we have $F=\{x\}^{c}$ for some $x\in[s+t-1]$, which implies that $F\in\mathcal{F}_{1}$, a contraction. If $F^{c}$ is not a singleton set, then from $F\notin\mathcal{G}_{1}$ we can deduce that $F^{c}\not\subseteq[2s+t-1]$. Thus there exists some $x\in[2s+t,n]$ with $x\notin{F}$. Now we consider the distinct sets $Y_{1},\dots,Y_{s-1}\in\mathcal{L}^{s}(A)\subseteq\mathcal{F}_{2}$ of size $n-s-1$. By definition the last increment set of $\mathcal{L}^{s}(A)$ is $[s+t+1,2s+t-1]$. This implies $A\cup[2s+t,n]\subseteq{Y_{i}}$ for $i\in[s-1]$, meaning that $F$ is not a superset of any of $Y_{1},\dots,Y_{s-1}$. Moreover, as $F$ has the larger size, $F$ is not a subset of any of $Y_{1},\dots,Y_{s-1}$. Hence, $F,Y_{1},\dots,Y_{s-1}$ are pairwise incomparable and will serve as the upper layer of our desired $\mathcal{K}_{s,t}$. We construct the lower layer as in the preceding case, namely, choosing $\{x_{1}\},\dots,\{x_{t}\}\in\mathcal{F}_{4}$ with $x_{i}\in{A}$ for all $i\in[t]$, see Figure~\ref{fig9}.
\end{proof}

\begin{lemma}
\label{lem7}
The family $\mathcal{F}$ is induced $\mathcal{K}_{s,t}$-saturated.
\end{lemma}
\begin{proof}
We see already in the preceding subsection that $\mathcal{F}$ is induced $\mathcal{K}_{s,t}$-free. For any $F\in2^{[n]}\backslash\mathcal{F}$, by considering the cases in Lemmas~\ref{lem4},~\ref{lem5}, and~\ref{lem6}, we conclude that $\mathcal{F}\cup\{F\}$ contains an induced copy of $\mathcal{K}_{s,t}$. Therefore, $\mathcal{F}$ is induced $\mathcal{K}_{s,t}$-saturated.
\end{proof}

\subsection{Proof of Theorem~\ref{thm1}}
\begin{proof}[Proof of Theorem~\ref{thm1}]
For given $n,s,t\in\mathbb{N}$ with $s\geq{t}\geq2$ and $n\geq2s+t-1$, we construct the family $\mathcal{F}=\bigcup_{i=1}^{5}\mathcal{F}_{i}\subseteq2^{[n]}$ as aforementioned. Then by Lemma~\ref{lem7}, $\mathcal{F}$ is an induced $\mathcal{K}_{s,t}$-saturated family. Accordingly, 
\begin{equation*}
\mathrm{sat}^{*}(n,\mathcal{K}_{s,t})\leq\lvert\mathcal{F}\rvert\leq\sum_{i=1}^{4}\lvert\mathcal{F}_{i}\rvert+\lvert\mathcal{F}_{5}\rvert\leq\left(\binom{s+t-1}{t}(s-1)+\binom{s+t-1}{t-1}(t-1)\right)n+c_{s,t},
\end{equation*}
where $c_{s,t}$ is upper bounded by the size of $\mathcal{F}_{5}$ that only depends on $s$ and $t$.
\end{proof}

\section{Proof of Theorem~\ref{thm3}}
\label{proofofthm3}
Let $\mathcal{P}$ be a poset, which contains two incomparable elements $a,b\in\mathcal{P}$, such that $a$ and $b$ are smaller than every element in $\mathcal{P}\backslash\{a,b\}$. Recall that such elements $a$ and $b$ are referred to as the legs of $\mathcal{P}$. Moreover, we call $\mathcal{P}\backslash\{a,b\}$ the \textit{body} of $\mathcal{P}$.

Let $n\in\mathbb{N}$ and fix an arbitrary $\mathcal{F}\subseteq2^{[n]}$ that is induced $\mathcal{P}$-saturated. To prove Theorem~\ref{thm3}, it suffices to show that $\lvert\mathcal{F}\backslash\{\emptyset\}\rvert\geq{n}$. Indeed, because $\mathcal{P}$ has two legs, no induced copy of $\mathcal{P}$ would contain $\emptyset$. This implies that $\emptyset\in\mathcal{F}$, namely, $\lvert\mathcal{F}\rvert=\lvert\mathcal{F}\backslash\{\emptyset\}\rvert+1\geq{n+1}$.

If $\mathcal{F}$ contains all singletons, we immediately obtain $\lvert\mathcal{F}\backslash\{\emptyset\}\rvert\geq{n}$. Otherwise, for every singleton set $\{x\}\in2^{[n]}\backslash\mathcal{F}$, adding $\{x\}$ to $\mathcal{F}$ creates at least one induced copy of $\mathcal{P}$, denoted by $\mathcal{P}'$, and $\mathcal{P}'$ must contain $\{x\}$. Moreover, the singleton set $\{x\}$ must be one of the legs of $\mathcal{P}'$, as the only proper subset of $\{x\}$ is $\emptyset$. We say that $C\in\mathcal{F}$ is a \textit{partner} of $x$ if $\{x\}$ and $C$ form the legs of some induced copy of $\mathcal{P}$ in $\mathcal{F}\cup\left\{\{x\}\right\}$. Let $C_{x}$ be one of the largest partners of $x$, chosen with respect to their cardinality. In particular, $x\notin{C_{x}}$ holds since $\{x\}$ and $C_{x}$ are incomparable.

\begin{lemma}
\label{lem8}
Let $\{x\}\in2^{[n]}\backslash\mathcal{F}$, then $C_{x}\cup\{x\}\in\mathcal{F}\backslash\{\emptyset\}$.
\end{lemma}
\begin{proof}
First, due to the way we defined $C_{x}$, there exists an induced copy of $\mathcal{P}$ in $\mathcal{F}\cup\left\{\{x\}\right\}$, denoted by $\mathcal{P}_{1}$, whose legs are $\{x\}$ and $C_{x}$. Suppose $C_{x}\cup\{x\}\notin\mathcal{F}\backslash\{\emptyset\}$. Then adding $C_{x}\cup\{x\}$ to $\mathcal{F}$ creates an induced copy of $\mathcal{P}$, denoted by $\mathcal{P}_{2}$.\\
\textbf{Case 1:} $C_{x}\cup\{x\}$ lies in the body of $\mathcal{P}_{2}$.\\
Then the legs of $\mathcal{P}_{2}$ are subsets of $C_{x}\cup\{x\}$. Since $\{x\}$ and $C_{x}$ are the legs of $\mathcal{P}_{1}$, every set in the body of $\mathcal{P}_{1}$ is a superset of $C_{x}\cup\{x\}$. Therefore, the body of $\mathcal{P}_{1}$ and the legs of $\mathcal{P}_{2}$ together form an induced copy of $\mathcal{P}$ in $\mathcal{F}$, a contradiction.\\
\textbf{Case 2:} $C_{x}\cup\{x\}$ is a leg of $\mathcal{P}_{2}$.\\
Let $D\in\mathcal{F}$ be the other leg of $\mathcal{P}_{2}$, note that $C_{x}\cup\{x\}$ and $D$ are incomparable. If $C_{x}\nsubseteq{D}$, then we can replace the legs of $\mathcal{P}_{2}$ with $C_{x}$ and $D$, from which we obtain an induced copy of $\mathcal{P}$ in $\mathcal{F}$, a contradiction. If $C_{x}\subseteq{D}$, given that $D\nsubseteq{C_{x}}\cup\{x\}$, we have $x\notin{D}$ and $\lvert{D}\rvert>\lvert{C_{x}}\rvert$. Hence, we can replace the legs of $\mathcal{P}_{2}$ with $\{x\}$ and $D$, and obtain an induced copy of $\mathcal{P}$ in $\mathcal{F}\cup\{\{x\}\}$. But then $D$ is a larger partner of $x$ than $C_{x}$, contradicting the way we choose $C_{x}$.
\end{proof}

Now we define the function
\begin{equation*}
f:\,[n]\to\mathcal{F}\backslash\{\emptyset\},\,f(x):=\left\{
\begin{array}{ll}
\{x\}&\text{if }\{x\}\in\mathcal{F},\\
C_{x}\cup\{x\}&\text{if }\{x\}\notin\mathcal{F}.
\end{array}\right.
\end{equation*}

\begin{lemma}
\label{lem9}
$f$ is an injective function.
\end{lemma}
\begin{proof}
Take arbitrary $x,y\in[n]$ with $x\neq{y}$, we split our proof into three cases.\\
\textbf{Case 1:} $\{x\},\{y\}\in\mathcal{F}$.\\
By definition we have $f(x)=\{x\}\neq\{y\}=f(y)$.\\
\textbf{Case 2:} $\{x\}\notin\mathcal{F},\,\{y\}\in\mathcal{F}$.\\
Then it follows that $\lvert{f(x)}\rvert=\lvert{C_{x}\cup\{x\}}\rvert>1$ and $\lvert{f(y)}\rvert=\lvert\{y\}\rvert=1$, hence, $f(x)\neq{f(y)}$.\\
\textbf{Case 3:} $\{x\},\{y\}\notin\mathcal{F}$.\\
We suppose $f(x)=f(y)$, namely, $C_{x}\cup\{x\}=C_{y}\cup\{y\}$. Recall that $x\notin{C_{x}}$ and $y\notin{C_{y}}$, $C_{x}$ and $C_{y}$ must be distinct and of the same size. By the choice of $C_{x}$, there exists an induced copy of $\mathcal{P}$ in $\mathcal{F}\cup\left\{\{x\}\right\}$, denoted by $\mathcal{P}_{1}$, whose legs are $\{x\}$ and $C_{x}$. Since $C_{y}\subseteq{C_{x}}\cup\{x\}$, $C_{y}$ is also a subset of all members from the body of $\mathcal{P}_{1}$. Note that as $C_{x}$ and $C_{y}$ are incomparable due to the same size, we can replace the legs of $\mathcal{P}_{1}$ with $C_{x}$ and $C_{y}$, which gives an induced copy of $\mathcal{P}$ in $\mathcal{F}$, a contradiction.
\end{proof}

\begin{proof}[Proof of Theorem~\ref{thm3}]
Let $\mathcal{F}\subseteq2^{[n]}$ be an induced $\mathcal{P}$-saturated family with the smallest size. Define the function $f:[n]\to\mathcal{F}\backslash\{\emptyset\}$ as above. By Lemmas~\ref{lem8} and~\ref{lem9}, $f$ is a well-defined injective function, which implies that $\lvert\mathcal{F}\backslash\{\emptyset\}\rvert\geq{n}$. Moreover, because no induced copy of $\mathcal{P}$ contains $\emptyset$, we must have $\emptyset\in\mathcal{F}$. Consequently, $\mathrm{sat}^{*}(n,\mathcal{P})=\lvert\mathcal{F}\rvert\geq{n+1}$.
\end{proof}

\begin{remark}
\label{rem2}
If the body of $\mathcal{P}$ is nonempty and $\mathcal{P}$ does not contain an element which is larger than all other elements in $\mathcal{P}$, then one can show that $\mathrm{sat}^{*}(n,\mathcal{P})\geq\min\{2^{n},n+2\}$ by defining $f$ to be a function from $[n]$ to $\mathcal{F}\backslash\{\emptyset,[n]\}$.
\end{remark}

\noindent\textbf{Acknowledgements.} The author would like to thank Let\'{i}cia Mattos, Silas Rathke, and Patricija Sapokaite for valuable discussions during EXCILL IV at UIUC. The author also thanks Maria Axenovich, Paul Bastide, and the anonymous referees for helpful comments on the manuscript.


\begin{thebibliography}{99}

\bibitem{bastide2023polynomial}
Paul Bastide, Carla Groenland, Maria-Romina Ivan, and Tom Johnston.
``A polynomial upper bound for poset saturation.''
\emph{European Journal of Combinatorics} \textbf{129} (2024): 103970.

\bibitem{bastide2023exact}
Paul Bastide, Carla Groenland, Hugo Jacob, and Tom Johnston.
``Exact antichain saturation numbers via a generalisation of a result of Lehman-Ron.''
\emph{Combinatorial Theory} \textbf{4(1)} (2024).

\bibitem{ferrara2017saturation}
Michael Ferrara, Bill Kay, Lucas Kramer, Ryan R. Martin, Benjamin Reiniger, Heather C. Smith, and Eric Sullivan.
``The saturation number of induced subposets of the Boolean lattice.''
\emph{Discrete Mathematics} \textbf{340(10)} (2017): 2479--2487.

\bibitem{freschi2023induced}
Andrea Freschi, Sim\'{o}n Piga, Maryam Sharifzadeh, and Andrew Treglown.
``The induced saturation problem for posets.''
\emph{Combinatorial Theory} \textbf{3(3)} (2023).

\bibitem{gerbner2013saturating}
D\'{a}niel Gerbner, Bal\'{a}zs Keszegh, Nathan Lemons, Cory Palmer, D\"{o}m\"{o}t\"{o}r P\'{a}lv\"{o}lgyi, and Bal\'{a}zs Patk\'{o}s.
``Saturating Sperner families.''
\emph{Graphs and Combinatorics} \textbf{29(5)} (2013): 1355--1364.

\bibitem{ivan2020saturation}
Maria-Romina Ivan.
``Saturation for the butterfly poset.''
\emph{Mathematika} \textbf{66(3)} (2020): 806--817.

\bibitem{ivan2022minimal}
Maria-Romina Ivan.
``Minimal diamond-saturated families.''
\emph{Contemporary Mathematics} \textbf{3(2)} (2022): 81--88.

\bibitem{keszegh2021induced}
Bal\'{a}zs Keszegh, Nathan Lemons, Ryan R. Martin, D\"{o}m\"{o}t\"{o}r P\'{a}lv\"{o}lgyi, and Bal\'{a}zs Patk\'{o}s.
``Induced and non-induced poset saturation problems.''
\emph{Journal of Combinatorial Theory, Series A} \textbf{184} (2021): 105497.

\bibitem{martin2020improved}
Ryan R. Martin, Heather C. Smith, and Shanise Walker.
``Improved bounds for induced poset saturation.''
\emph{The Electronic Journal of Combinatorics} \textbf{27(2)} (2020): P2.31.

\bibitem{martin2024saturation}
Ryan R. Martin and Nick Veldt.
``Saturation of $k$-chains in the Boolean lattice.''
\emph{The Electronic Journal of Combinatorics} \textbf{32(1)} (2025): P1.55.

\bibitem{martin2024induced}
Ryan R. Martin and Nick Veldt.
``Induced saturation of the poset $2C_{2}$.''
\emph{Annals of Combinatorics} (2025).

\bibitem{morrison2014sperner}
Natasha Morrison, Jonathan A. Noel, and Alex Scott.
``On saturated $k$-Sperner systems.''
\emph{The Electronic Journal of Combinatorics} \textbf{21(3)} (2014): P3.22.
\end{thebibliography}
\end{document}